\documentclass{amsart}

\newtheorem{theorem}{Theorem}
\newtheorem{lemma}[theorem]{Lemma}

\theoremstyle{definition}
\newtheorem{definition}{Definition}
\theoremstyle{remark}
\newtheorem{remark}{Remark}
\newtheorem{corollary}[remark]{Corollary}

\usepackage{amssymb}
\usepackage{empheq}
\usepackage{stmaryrd}
\usepackage{enumerate}
\usepackage{multicol}
\usepackage[shortlabels]{enumitem}
\usepackage{calc}
\usepackage{tikz}
\usepackage{listings}
\usepackage{bussproofs}
\usepackage[ruled,lined]{algorithm2e}

\lstdefinestyle{derivation} {
    breakatwhitespace=false,         
    breaklines=true,                 
    captionpos=b,                    
    keepspaces=true,                 
    numbers=left,                    
    numbersep=5pt,
    escapechar=\&
}

\newcommand{\addresseshere}{\enddoc@text\let\enddoc@text\relax}

\usetikzlibrary{automata,positioning}
\usetikzlibrary{arrows}

\usepackage[left=2.0cm,%
                right=2.0cm,%
                top=2.5cm,%
                bottom=2.5cm,%
                headheight=12pt,%
                a4paper]{geometry}%

\begin{document}

\title{Standard G\"odel modal logics are not realized by G\"odel justification logics}

\author{Nicholas Pischke}
\address{Hoch-Weiseler Str. 46, Butzbach, 35510, Hesse, Germany}
\email{pischkenicholas@gmail.com}


\begin{abstract}
We show that the standard G\"odel modal logics, as initially introduced by Caicedo and Rodriguez in \cite{CR2009,CR2010}, are not realized by the basic G\"odel justification logics although being related by the forgetful projection.
\end{abstract}

\maketitle

\section{Introduction}
A central concept in classical justification logics is their relation to the classical (necessity-based) modal logics. The heart of this relation is the so called realization theorem as already present in the earliest papers (see e.g. \cite{Art1995,Art2001}) on the topic: given a modal theorem $\phi$, the realization theorem provides a function $r$, mapping $\phi$ to a justification formula $\phi^r$, which assigns every occurrence of the modality $\Box$ in $\phi$ some justification term and this resulting $\phi^r$ is then again a theorem of the corresponding justification logic. Such a realization is called normal, if negative occurrences of modalities are realized by justification variables and positive occurrences of modalities are realized with polynomials in these variables. 

In this paper, we investigate this property for fuzzy variants of modal and justification logics, namely standard G\"odel modal logics as introduced by Caicedo and Rodriguez in \cite{CR2009,CR2010} and G\"odel justification logics as introduced by Ghari in \cite{Gha2014} and Pischke in \cite{Pis2018}. These variants replace the classical boolean base of classical justification logic with $[0,1]$-valued G\"odel logics, one of the three main t-norm based fuzzy logics in the sense of H\'ajek \cite{Haj1998}, and initially originating from an intuitionistic perspective along the lines of G\"odel \cite{Goe1932}, Dummett \cite{Dum1959} and Horn \cite{Hor1969}.

A first resulting difference is that, while in classical modal logic there is a natural notion of the dual operator $\Diamond\theta\equiv\neg\Box\neg\theta$, the natural dual of $\Box$ in standard G\"odel modal logic is not internally definable any more. This gives rise to three different fuzzy G\"odel modal logics, a bi-modal version containing both $\Box$ and $\Diamond$ (see \cite{CR2015}) and its respective $\Box$ and $\Diamond$-fragments (see \cite{CR2009,CR2010}).

We only consider the $\Box$-fragment in this paper as the justification modality ``$t:$``, in its standard semantical interpretation (both in classical and in G\"odel justification logic), is a necessity-style operator. As there is no immediate dual notion of ``$t:$``, neither in classical nor in G\"odel justification logic, there is also no immediate way of interpreting $\Diamond$ in the fuzzy justification setting.

Syntactically, we define a set of justification terms $Jt$ as
\[
Jt:t::= c\mid x\mid [t+t]\mid [t\cdot t]\mid\; !t\mid\; ?t
\]
where $c\in C:=\{c_i\mid i\in\mathbb{N}\}$ is a justification constant and $x\in V:=\{x_i\mid i\in\mathbb{N}\}$ is a justification variable. The corresponding language of justification logics $\mathcal{L}_J$ is then given as
\[
\mathcal{L}_J:\phi::= \bot\mid p\mid (\phi\rightarrow\phi)\mid (\phi\land\phi)\mid t:\phi
\]
with $t\in Jt$ and $p\in Var:=\{p_i\mid i\in\mathbb{N}\}$ a propositional variable. $\neg$ is introduced as a syntactical abbreviation by $\neg\phi:=(\phi\rightarrow\bot)$.\\

The fundamental semantics of G\"odel justification logics used here is the many-valued analogue of the classical Mkrtychev models \cite{Mkr1997}, called G\"odel-Mkrtychev models. These extend the standard minimum t-norm based semantics for propositional G\"odel logics as follows, where we denote the minimum t-norm by $\odot$ and by $\oplus$ the maximum function, i.e. $x\odot y=\min\{x,y\}$ and $x\oplus y=\max\{x,y\}$. Precisely, a G\"odel-Mkrtychev model is a structure $\mathfrak{M}=\langle\mathcal{E},e\rangle$ where
\begin{enumerate}
\item $\mathcal{E}:Jt\times\mathcal{L}_J\to[0,1]$,
\item $e:Var\to[0,1]$,
\end{enumerate}
and which satisfies
\begin{enumerate}[(i)]
\item $\mathcal{E}(t,\phi\rightarrow\psi)\odot\mathcal{E}(s,\phi)\leq\mathcal{E}(t\cdot s,\psi)$ for all $t,s\in Jt$, $\phi,\psi\in\mathcal{L}_J$,
\item $\mathcal{E}(t,\phi)\oplus\mathcal{E}(s,\phi)\leq\mathcal{E}(t+s,\phi)$ for all $t,s\in Jt$, $\phi\in\mathcal{L}_J$.
\end{enumerate}
We denote the class of all G\"odel-Mkrtychev models by $\mathsf{GM}$. We call a $\mathsf{GM}$-model $\mathfrak{M}=\langle\mathcal{E},e\rangle$ \emph{crisp} if both $\mathcal{E}$ and $e$ only take values in $\{0,1\}$. 

For a $\mathsf{GM}$-model $\mathfrak{M}=\langle\mathcal{E},e\rangle$, we define its evaluation function $\vert\cdot\vert_\mathfrak{M}:\mathcal{L}_J\to [0,1]$ as follows:
\begin{itemize}
\item $\vert\bot\vert_\mathfrak{M}=0$,
\item $\vert p\vert_\mathfrak{M}=e(p)$ for $p\in Var$,
\item $\vert\phi\rightarrow\psi\vert_\mathfrak{M}=\vert\phi\vert_\mathfrak{M}\Rightarrow\vert\psi\vert_\mathfrak{M}$,
\item $\vert\phi\land\psi\vert_\mathfrak{M}=\vert\phi\vert_\mathfrak{M}\odot\vert\psi\vert_\mathfrak{M}$,
\item $\vert t:\phi\vert_\mathfrak{M}=\mathcal{E}(t,\phi)$,
\end{itemize}
where we write $\Rightarrow$ for the residuum (see e.g. \cite{Haj1998}) of $\odot$, i.e.
\[
x\Rightarrow y=\begin{cases}y&\text{if }x>y\\1&\text{otherwise}\end{cases}
\]
for $x,y\in [0,1]$. For the derived connective $\neg$, we obtain the following derived truth function $\sim$:
\[
\sim x=\begin{cases}0&\text{if }x>0\\1&\text{otherwise}\end{cases}
\]
and for $\sim\sim x$, we also write $\sim^2x$.

We may extend the evaluation to sets of formulas $\Gamma\subseteq\mathcal{L}_J$ by setting $\vert\Gamma\vert_\mathfrak{M}=\inf_{\phi\in\Gamma}\{\vert\phi\vert_\mathfrak{M}\}$. We write $\mathfrak{M}\models\phi$ if $\vert\phi\vert_\mathfrak{M}=1$ and $\mathfrak{M}\models\Gamma$ if $\mathfrak{M}\models\phi$ for any $\phi\in\Gamma$.\\

A $\mathsf{GM}$-model $\mathfrak{M}=\langle\mathcal{E},e\rangle$ is called a
\begin{enumerate}
\item $\mathsf{GMT}$-model if $\mathcal{E}(t,\phi)\leq\vert\phi\vert_\mathfrak{M}$ for all $t\in Jt$, $\phi\in\mathcal{L}_J$,
\item $\mathsf{GM4}$-model if $\mathcal{E}(t,\phi)\leq\mathcal{E}(!t,t:\phi)$ for all $t\in Jt$, $\phi\in\mathcal{L}_J$,
\item $\mathsf{GMLP}$-model if (1) and (2),
\item $\mathsf{GM45}$-model if (2) and $\sim\mathcal{E}(t,\phi)\leq\mathcal{E}(?t,\neg t:\phi)$ for all $t\in Jt$, $\phi\in\mathcal{L}_J$,
\item $\mathsf{GMT45}$-model if (1) and (4).
\end{enumerate}
\begin{definition}
Let $\mathsf{C}$ be class of $\mathsf{GM}$-models. For $\Gamma\cup\{\phi\}\subseteq\mathcal{L}_J$, we say that $\Gamma$ 1-entails $\phi$ in $\mathsf{C}$, written $\Gamma\models_\mathsf{C}\phi$, if for any model $\mathfrak{M}\in\mathsf{C}$, if $\mathfrak{M}\models\Gamma$, then $\mathfrak{M}\models\phi$.
\end{definition}
The other standard semantics for classical justification logics defined by so called Fitting-models, see \cite{Fit2005}, also extends to the fuzzy cases, see e.g. \cite{Gha2014,Gha2016,Pis2018}.

We define the following proof systems for G\"odel justification logic over $\mathcal{L}_J$ based on H\'ajek's strongly complete Hilbert-style proof calculus for propositional G\"odel logic given in \cite{Haj1998}:
\begin{definition}
The Hilbert-style calculus $\mathcal{GJ}_0$ is given by the following axiom schemes and rules over $\mathcal{L}_J$:
\begin{description}
\item [($A1$)] $(\phi\rightarrow\psi)\rightarrow((\psi\rightarrow\chi)\rightarrow(\phi\rightarrow\chi))$
\item [($A2$)] $(\phi\land\psi)\rightarrow\phi$
\item [($A3$)] $(\phi\land\psi)\rightarrow(\psi\land\phi)$
\item [($A5a$)] $(\phi\rightarrow(\psi\rightarrow\chi))\rightarrow((\phi\land\psi)\rightarrow\chi)$
\item [($A5b$)] $((\phi\land\psi)\rightarrow\chi)\rightarrow(\phi\rightarrow(\psi\rightarrow\chi))$
\item [($A6$)] $((\phi\rightarrow\psi)\rightarrow\chi)\rightarrow(((\psi\rightarrow\phi)\rightarrow\chi)\rightarrow\chi)$
\item [($A7$)] $\bot\rightarrow\phi$
\item [($G4$)] $\phi\rightarrow(\phi\land\phi)$
\item [($J$)] $t:(\phi\rightarrow\psi)\rightarrow (s:\phi\rightarrow [t\cdot s]:\psi)$
\item [($+$)] $t:\phi\rightarrow [t+s]:\phi$, $s:\phi\rightarrow [t+s]:\phi$
\item [($MP$)] From $\phi\rightarrow\psi$ and $\phi$, infer $\psi$.
\end{description}
By $\mathcal{G}$, we denote the fragment without the axiom schemes ($J$) and ($+$). We then define the following axiomatic extensions of $\mathcal{GJ}_0$:
\begin{enumerate}
\item $\mathcal{GJT}_0$ is the extension of $\mathcal{GJ}_0$ by the scheme $(F):t:\phi\rightarrow\phi$,
\item $\mathcal{GJ}4_0$ is the extension of $\mathcal{GJ}_0$ by the scheme $(!):t:\phi\rightarrow !t:t:\phi$,
\item $\mathcal{GLP}_0$ is the extension of $\mathcal{GJT}_0$ by the scheme $(!)$,
\item $\mathcal{GJ}45_0$ is the extension of $\mathcal{GJ}4_0$ by the scheme $(?):\neg t:\phi\rightarrow ?t:\neg t:\phi$,
\item $\mathcal{GJT}45_0$ is the extension of $\mathcal{GJ}45_0$ by the scheme $(F)$.
\end{enumerate}
\end{definition}
Provability (possibly with assumptions) in such a Hilbert-style calculus $\mathcal{S}$ is defined as usual and denoted by the relation symbol $\vdash_\mathcal{S}$.

Let $\mathcal{GJL}_0$ be one the previously introduced G\"odel justification logics. We call a set $CS$ of formulas of the form
\[
c_{i_n} :\dots:c_{i_1}:\phi, \quad n\geq 1,
\]
where $\phi$ is an axiom instance of $\mathcal{GJL}_0$, $c_{i_k}\in C$ and such that 
\[
\text{if }c_{i_n} :\dots:c_{i_1}:\phi\in CS\text{, then }c_{i_k} :\dots:c_{i_1}:\phi\text{ for every }k\in\{1,\dots,n\},
\]
a \emph{constant specification} for $\mathcal{GJL}_0$. A constant specification $CS$ for $\mathcal{GJL}_0$ is called \emph{axiomatically appropriate} if for every axiom instance $\phi$ of $\mathcal{GJL}_0$, there is a constant $c\in C$ such that $c:\phi\in CS$ and
\[
\text{if }c_{i_n} :\dots:c_{i_1}:\phi\text{, then }c_{i_{n+1}}:c_{i_n} :\dots:c_{i_1}:\phi\text{ for some constant }c_{i_{n+1}}.
\]
A constant specification $CS$ for $\mathcal{GJL}_0$ is called total if
\[
c_{i_n} :\dots:c_{i_1}:\phi\in CS\text{ for every }n\geq 1,i_1,\dots,i_n\in\mathbb{N}\text{ and every axiom instace }\phi.
\]
Given a constant specification $CS$ for $\mathcal{GJL}_0$, we define the logic $\mathcal{GJL}_{CS}$ as the extension of $\mathcal{GJL}_0$ by the rule ($CS$):
\[
\text{From }c:\phi\in CS\text{, infer }c:\phi.
\]
We say that a G\"odel-Mkrtychev model $\mathfrak{M}=\langle\mathcal{E},e\rangle$ respects a constant specification $CS$ if $\mathcal{E}(c,\phi)=1$ for every $c:\phi\in CS$. Given a class of G\"odel-Mkrtychev models $\mathsf{C}$, we denote the subclass of all models respecting $CS$ by $\mathsf{C_{CS}}$.
A first important result on the G\"odel-based systems is the lifting lemma as an analogue to the classical case.
\begin{lemma}[Lifting lemma, P. \cite{Pis2018}]\label{lem:gjlcslifting}
Let
\[
\mathcal{GJL}_0\in\{\mathcal{GJ}_0,\mathcal{GJT}_0,\mathcal{GJ}4_0,\mathcal{GLP}_0,\mathcal{GJ}45_0,\mathcal{GJT}45_0\}
\]
and $CS$ be an axiomatically appropriate constant specification for $\mathcal{GJL}_0$. If\\$\{\psi_1,\dots,\psi_n\}\vdash_{\mathcal{GJL}_{CS}}\phi$, then for any justification terms $t_1,\dots,t_n\in Jt$, there is a justification term $t\in Jt$ such that
\[
\{t_1:\psi_1,\dots,t_n:\psi_n\}\vdash_{\mathcal{GJL}_{CS}}t:\phi.
\]
\end{lemma}
A direct consequence of the lifting lemma is the internalization property for certain justification logics.
\begin{corollary}[Internalization]\label{cor:internalization}
Let
\[
\mathcal{GJL}_0\in\{\mathcal{GJ}_0,\mathcal{GJT}_0,\mathcal{GJ}4_0,\mathcal{GLP}_0,\mathcal{GJ}45_0,\mathcal{GJT}45_0\}
\]
and $CS$ be an axiomatically appropriate constant specification for $\mathcal{GJL}_0$. Then if $\vdash_{\mathcal{GJ}_{CS}}\phi$, then there is a $t\in Jt$ such that $\vdash_{\mathcal{GJ}_{CS}}t:\phi$.
\end{corollary}
The main theorem on G\"odel justification logics used in this paper is the completeness theorem for the above systems and the G\"odel-Mkrtychev models introduced before.
\begin{theorem}[Completeness, P. \cite{Pis2018}]\label{thm:gmjlcscomp}
Let
\[
\mathcal{GJL}_0\in\{\mathcal{GJ}_0,\mathcal{GJT}_0,\mathcal{GJ}4_0,\mathcal{GLP}_0,\mathcal{GJ}45_0,\mathcal{GJT}45_0\}
\]
and $CS$ be a constant specification for $\mathcal{GJL}_0$. Let
\[
\mathsf{GMJL}\in\{\mathsf{GM},\mathsf{GMT},\mathsf{GM4},\mathsf{GMLP},\mathsf{GM45},\mathsf{GMT45}\}
\]
be the corresponding class of G\"odel-Mkrtychev models for $\mathcal{GJL}_0$. Then for all $\Gamma\cup\{\phi\}\subseteq\mathcal{L}_J$:
\[
\Gamma\vdash_{\mathcal{GJL}_{CS}}\phi\text{ iff }\Gamma\models_\mathsf{GMJL_{CS}}\phi.
\] 
\end{theorem}
On the modal side, we fix a necessity-based modal language $\mathcal{L}_\Box$ by
\[
\mathcal{L}_\Box:\phi::=\bot\mid p\mid (\phi\land\phi)\mid (\phi\rightarrow\phi)\mid \Box\phi
\]
Before we concern ourselves with the concept of realizability, we present the standard Hilbert-style proof theoretic access to standard G\"odel modal logics based on the work of Caicedo and Rodriguez in \cite{CR2009,CR2010}. We define the following Hilbert-style proof calculi in the language $\mathcal{L}_\Box$:
\begin{definition}
$\mathcal{GK}_\Box$ is given by the following axiom schemes and rules:
\begin{description}
\item [($G$)] The axiom schemes of the calculus $\mathcal{G}$.
\item [($K$)] $\Box(\phi\rightarrow\psi)\rightarrow (\Box\phi\rightarrow\Box\psi)$
\item [($Z$)] $\neg\neg\Box\phi\rightarrow\Box\neg\neg\phi$
\item [($MP$)] From $\phi\rightarrow\psi$ and $\phi$, infer $\psi$.
\item [($N\Box$)] From $\vdash\phi$, infer $\vdash\Box\phi$.
\end{description}
We then define the following axiomatic extensions of $\mathcal{GK}_\Box$:
\begin{enumerate}
\item $\mathcal{GT}_\Box$ is the extension of $\mathcal{GK}_\Box$ by the axiom scheme $(T):\Box\phi\rightarrow\phi$,
\item $\mathcal{GK}4_\Box$ is the extension of $\mathcal{GK}_\Box$ by the axiom scheme $(4):\Box\phi\rightarrow\Box\Box\phi$,
\item $\mathcal{GS}4_\Box$ is the extension of $\mathcal{GT}_\Box$ by the axiom scheme $(4)$.
\end{enumerate}
\end{definition}
The notation of the rule ($N\Box$) is used to indicate that it may only be applied to pure theorems of the respective calculus if it is a proof with a set of assumptions.

In \cite{CR2010}, Caicedo and Rodriguez obtained completeness theorems for these logics together with a natural semantics defined over model classes of $[0,1]$-valued Kripke models, called G\"odel-Kripke models.

Justification logics are often presented relative to a given constant specification. It shall be noted that by the modal inference rule ($N\Box$) in all of the above systems, for every theorem $\theta$, $\Box\theta$ is a theorem as well. Thus, any candidate G\"odel justification logic for realization of a corresponding standard G\"odel modal logic has to have the internalization property.

In the following, for a given proof system $\mathcal{S}$ over a language $\mathcal{L}$, we write $Th_\mathcal{S}=\{\phi\in\mathcal{L}\mid\;\vdash_\mathcal{S}\phi\}$.
\subsection{Forgetful projection}
A natural projection from the explicit modal language $\mathcal{L}_J$ to $\mathcal{L}_\Box$ is the one mapping every explicit modality ``$t:$`` to the unexplicit modality $\Box$, called the forgetful projection. We define the forgetful projection operator $\circ:\mathcal{L}_J\to\mathcal{L}_\Box$ formally by recursion on the structure of $\mathcal{L}_J$ as follows:
\begin{itemize}
\item $p\mapsto p$ for $p\in Var$,
\item $\bot\mapsto\bot$, $\top\mapsto\top$,
\item $\phi\land\psi\mapsto\phi^\circ\land\psi^\circ$,
\item $\phi\rightarrow\psi\mapsto\phi^\circ\rightarrow\psi^\circ$,
\item $t:\phi\mapsto\Box\phi^\circ$
\end{itemize}
We may extend $\circ$ to sets of formulas $\Gamma\subseteq\mathcal{L}_J$ via $\Gamma^\circ:=\{\phi^\circ\mid\phi\in\Gamma\}$.
\begin{remark}
For the various axioms of G\"odel justification logics, we obtain the following forgetful projections:
\begin{enumerate}
\item $(t:(\phi\rightarrow\psi)\rightarrow(s:\phi\rightarrow [t\cdot s]:\psi))^\circ=\Box(\phi^\circ\rightarrow\psi^\circ)\rightarrow (\Box\phi^\circ\rightarrow\Box\psi^\circ)$,
\item $(t:\phi\rightarrow [t+s]:\phi)^\circ=\Box\phi^\circ\rightarrow\Box\phi^\circ$, $(s:\phi\rightarrow [t+s]:\phi)^\circ=\Box\phi^\circ\rightarrow\Box\phi^\circ$,
\item $(t:\phi\rightarrow\phi)^\circ=\Box\phi^\circ\rightarrow\phi^\circ$,
\item $(t:\phi\rightarrow !t:t:\phi)^\circ=\Box\phi^\circ\rightarrow\Box\Box\phi^\circ$,
\end{enumerate}
\end{remark}
Note, that the cases in (2) are instances of a propositional tautology, while (1),(3),(4) and (5) are instances of the various axioms of standard G\"odel modal logic, all in the language of $\mathcal{L}_\Box$.
\begin{theorem}\label{thm:forgetfullprojection}
Let $\mathcal{GJL}_0\in\{\mathcal{GJ}_0,\mathcal{GJT}_0,\mathcal{GJ}4_0,\mathcal{GLP}_0\}$, $CS$ be a constant specification for $\mathcal{GJL}_0$ and $\mathcal{GML}_\Box\in\{\mathcal{GK}_\Box,\mathcal{GT}_\Box,\mathcal{GK}4_\Box,\mathcal{GS}4_\Box\}$ be the corresponding G\"odel modal logic. Then, for all $\Gamma\cup\{\phi\}\subseteq\mathcal{L}_J$: $\Gamma\vdash_{\mathcal{GJL}_{CS}}\phi$ implies $\Gamma^\circ\vdash_{\mathcal{GML}_\Box}\phi^\circ$.
\end{theorem}
The proof of the theorem is a straightforward induction on the length of the proof.
\section{Realization fails without factivity}
In the following, we concern ourselves with the non-realizability of the axiom scheme ($Z$). We approach this using a countermodel construction, where we for now require that the justification logics do not contain the factivity axiom scheme ($F$).

For this, let $TCS$ be the total constant specification for $\mathcal{GJ}45_0$, and $x\in (0,1)$. We define the \emph{$x$-rooted} provability model $\mathfrak{M}_x=\langle\mathcal{E}_x,e_x\rangle$ by $e_x(p)=x$ for any $p\in Var$ and
\[
\mathcal{E}_x(t,\phi)=\begin{cases}1 &\text{if }\vdash_{\mathcal{GJ}45_{TCS}}\phi\text{ and }\vdash_{\mathcal{GJ}45_{TCS}}t:\phi\\x &\text{else}\end{cases}
\]
for any $t\in Jt,\phi\in\mathcal{L}_J$. It is easy to see that for $\phi\in\mathcal{L}_J$, we have $\vert\phi\vert_{\mathfrak{M}_x}\in\{0,x,1\}$. We then first obtain the following:
\begin{lemma}\label{lem:MxWellDef}
For any $x\in (0,1)$, $\mathfrak{M}_x$ is a $\mathsf{GM45_{TCS}}$-model.
\end{lemma}
\begin{proof}
We verify the conditions:
\begin{enumerate}
\item We have $\vdash_{\mathcal{GJ}45_{TCS}}c:\phi$ for any $c:\phi\in TCS$. By definition, either $\phi$ is an axiom instance or $\phi=d:\psi\in TCS$ by downward closure. Either way $\vdash_{\mathcal{GJ}45_{TCS}}\phi$ and thus we have $\mathcal{E}_x(c,\phi)=1$ for any such $c:\phi$, i.e. $\mathfrak{M}_x$ respects $TCS$.
\item Let  $\phi\in\mathcal{L}_J$ and $t,s\in Jt$. If $\mathcal{E}_x(t,\phi)\oplus\mathcal{E}_x(s,\phi)=x$, then the claim is immediate. Thus suppose $\mathcal{E}_x(t,\phi)\oplus\mathcal{E}_x(s,\phi)=1$, i.e. per definition $\mathcal{E}_x(t,\phi)=1$ or $\mathcal{E}_x(s,\phi)=1$. In either case $\vdash_{\mathcal{GJ}45_{TCS}}\phi$ and additionally $\vdash_{\mathcal{GJ}45_{TCS}}t:\phi$ or $\vdash_{\mathcal{GJ}45_{TCS}}s:\phi$. Either way, by the axiom scheme ($+$) and the rule ($MP$), we have $\vdash_{\mathcal{GJ}45_{TCS}}[t+s]:\phi$, i.e. $\mathcal{E}_x(t+s,\phi)=1$.
\item Let $\phi,\psi\in\mathcal{L}_J$ and $t,s\in Jt$. If $\mathcal{E}_x(t,\phi\rightarrow\psi)\odot\mathcal{E}_x(s,\phi)=x$, then the claim is immediate. Thus suppose $\mathcal{E}_x(t,\phi\rightarrow\psi)\odot\mathcal{E}_x(s,\phi)=1$, i.e. $\mathcal{E}_x(t,\phi\rightarrow\psi)$, $\mathcal{E}_x(s,\phi)=1$ and thus $\vdash_{\mathcal{GJ}45_{TCS}}\phi\rightarrow\psi$, $\vdash_{\mathcal{GJ}45_{TCS}}t:(\phi\rightarrow\psi)$ as well as $\vdash_{\mathcal{GJ}45_{TCS}}\phi$ and $\vdash_{\mathcal{GJ}45_{TCS}}s:\phi$. By ($MP$) and the axiom scheme ($J$), we have $\vdash_{\mathcal{GJ}45_{TCS}}\psi$ and $\vdash_{\mathcal{GJ}45_{TCS}}[t\cdot s]:\psi$, i.e. $\mathcal{E}_x(t\cdot s,\psi)=1$.
\item Let $\phi\in\mathcal{L}_J$ and $t\in Jt$. If $\mathcal{E}_x(t,\phi)=x$, then we have immediately that $\mathcal{E}_x(t,\phi)\leq\mathcal{E}_x(!t,t:\phi)$. Thus, suppose $\mathcal{E}_x(t,\phi)=1$, then $\vdash_{\mathcal{GJ}45_{TCS}}\phi$ and $\vdash_{\mathcal{GJ}45_{TCS}}t:\phi$. The latter implies $\vdash_{\mathcal{GJ}45_{TCS}}!t:t:\phi$ by the axiom scheme ($!$) and ($MP$), i.e. $\mathcal{E}_x(!t,t:\phi)=1$.
\item We always have $\mathcal{E}_x(t,\phi)\in\{x,1\}$, i.e. as $x>0$ we always have $\sim\mathcal{E}_x(t,\phi)=0$ and thus, for any $\phi\in\mathcal{L}_J$ and $t\in Jt$, we have $\sim\mathcal{E}_x(t,\phi)\leq\mathcal{E}_x(?t,\neg t:\phi)$.
\end{enumerate}
\end{proof}
$\mathfrak{M}_x$ now serves as a counter model for realization instances of the modal axiom ($Z$).
\begin{lemma}\label{lem:RealCounterModel}
For any $\phi\in\mathcal{L}_J$ such that $\not\vdash_{\mathcal{GJ}45_{TCS}}\neg\neg\phi$ and any $t,s\in Jt$:
\[
\not\vdash_{\mathcal{GJ}45_{TCS}}\neg\neg t:\phi\rightarrow s:\neg\neg\phi.
\]
\end{lemma}
\begin{proof}
Suppose $\not\vdash_{\mathcal{GJ}45_{CS}}\neg\neg\phi$ for $\phi\in\mathcal{L}_J$ and let $t,s\in Jt$ as well as $x\in (0,1)$. As 
\[
\vdash_{\mathcal{GJ}45_{TCS}}\phi\rightarrow\neg\neg\phi,
\]
we have $\not\vdash_{\mathcal{GJ}45_{TCS}}\phi$ as otherwise $\vdash_{\mathcal{GJ}45_{TCS}}\neg\neg\phi$ by ($MP$). Thus, $\vert t:\phi\vert_{\mathfrak{M}_x}=\mathcal{E}_x(t,\phi)=x\in (0,1)$.

As $\mathcal{E}_x(t,\phi)>0$, we have $\vert\neg\neg t:\phi\vert_{\mathfrak{M}_x}=1$ by the semantical evaluation of $\neg$ by $\sim$. However, we have 
\[
\vert s:\neg\neg\phi\vert_{\mathfrak{M}_x}=\mathcal{E}_x(s,\neg\neg\phi)=x<1
\]
as $\not\vdash_{\mathcal{GJ}45_{TCS}}\neg\neg\phi$. Thus, we have
\[
\vert\neg\neg t:\phi\rightarrow s:\neg\neg\phi\vert_{\mathfrak{M}_x}=x<1
\]
and by Lem. \ref{lem:MxWellDef} $\mathfrak{M}_x$ is a $\mathsf{GM45_{TCS}}$-model. Per definition, we have
\[
\not\models_{\mathsf{GM45}_{TCS}}\neg\neg t:\phi\rightarrow s:\neg\neg\phi,
\]
that is by Thm. \ref{thm:gmjlcscomp}:
\[
\not\vdash_{\mathcal{GJ}45_{TCS}}\neg\neg t:\phi\rightarrow s:\neg\neg\phi.
\]
\end{proof}
By this lemma, for any formula for which its double-negation projection is not provable (or valid), there is no valid (realized) formula structured like the ($Z$)-axiom. As for e.g. any propositional variable $p$, its double negation is never provable, we have the following two theorems.
\begin{theorem}
For any constant specification $CS$ for $\mathcal{GJ}_0$: $(Th_{\mathcal{GJ}_{CS}})^\circ\subsetneq Th_{\mathcal{GK}_\Box}$.
\end{theorem}
\begin{proof}
$(Th_{\mathcal{GJ}_{CS}})^\circ\subseteq Th_{\mathcal{GK}_\Box}$ follows from Thm. \ref{thm:forgetfullprojection}. By the modal axiom ($Z$), $\vdash_{\mathcal{GK}_\Box}\neg\neg\Box p\rightarrow\Box\neg\neg p$ for $p\in Var$, but as $\not\vdash_{\mathcal{GJ}45_{TCS}}\neg\neg p$, for any $t,s\in Jt$:
\[
\not\vdash_{\mathcal{GJ}45_{TCS}}\neg\neg t:p\rightarrow s:\neg\neg p,
\]
by Lem. \ref{lem:RealCounterModel}. Thus, by $CS\subseteq TCS$ as all $\mathcal{GJ}_0$ axiom schemes are also $\mathcal{GJ}45_0$ axiom schemes, we have 
\[
\not\vdash_{\mathcal{GJ}_{CS}}\neg\neg t:p\rightarrow s:\neg\neg p
\]
for any $t,s\in Jt$ as if there would be a proof, this proof could be also carried out in $\mathcal{GJ}45_{TCS}$. Thus, there is no $\phi\in\mathcal{L}_J$ such that $\vdash_{\mathcal{GJ}_{CS}}\phi$, i.e. $\phi\in Th_{\mathcal{GJ}_{CS}}$, and such that $\phi^\circ=\neg\neg\Box p\rightarrow\Box\neg\neg p$.
\end{proof}
By a similar proof, we have the following.
\begin{theorem}
For any constant specification $CS$ for $\mathcal{GJ}4_0$: $(Th_{\mathcal{GJ}4_{CS}})^\circ\subsetneq Th_{\mathcal{GK}4_\Box}$.
\end{theorem}
In fact, $\mathcal{GJ}4_{CS}$ does not even realize $\mathcal{GK}_\Box$, as the problem remains with axiom ($Z$). However, of course the forgetful projection of the introspection axiom scheme $t:\phi\rightarrow !t:t:\phi$ is not a theorem of $\mathcal{GK}_\Box$, i.e. $(Th_{\mathcal{GJ}4_{CS}})^\circ\not\subseteq Th_{\mathcal{GK}_\Box}$.

It also important to note that it is crucial for the proof of Thm. \ref{lem:RealCounterModel} that G\"odel-Mkrtychev models are many-valued as $x\in (0,1)$ is necessary. Making $\mathfrak{M}_x$ crisp by moving $x$ to $1$ makes any instance of $\neg\neg t:\phi\rightarrow s:\neg\neg\phi$ valid in $\mathfrak{M}_x$ and moving $x$ to $0$ makes at least some instance of $\neg\neg t:\phi\rightarrow s:\neg\neg\phi$ valid in $\mathfrak{M}_x$ for any $\phi$:

As we have that $\vdash_{\mathcal{GJ}45_{TCS}}\phi\rightarrow\neg\neg\phi$, by internalization (Corr. \ref{cor:internalization}, as $TCS$ is axiomatically appropriate) it follows, that we have $\vdash_{\mathcal{GJ}45_{TCS}}r:(\phi\rightarrow\neg\neg\phi)$ for some $r\in Jt$. Thus by ($J$) and modus ponens, we have $\vdash_{\mathcal{GJ}45_{TCS}}t:\phi\rightarrow[r\cdot t]:\neg\neg\phi$ for any $t\in Jt$. Thus, we have:
\begin{itemize}
\item If $\mathcal{E}_0(t,\phi)=0$, then $\sim^2\mathcal{E}_0(t,\phi)=0$ and there is nothing to show.
\item If $\mathcal{E}_0(t,\phi)=1$, then $\sim^2\mathcal{E}_0(t,\phi)=1$ and by definition $\vdash_{\mathcal{GJ}45_{TCS}}\phi$ and $\vdash_{\mathcal{GJ}45_{TCS}}t:\phi$. Then, with $\vdash_{\mathcal{GJ}45_{TCS}}t:\phi\rightarrow[r\cdot t]:\neg\neg\phi$ and $\vdash_{\mathcal{GJ}45_{TCS}}\phi\rightarrow\neg\neg\phi$, by ($MP$) we have $\vdash_{\mathcal{GJ}45_{TCS}}\neg\neg\phi$ and $\vdash_{\mathcal{GJ}45_{TCS}}[r\cdot t]:\neg\neg\phi$, i.e. $\mathcal{E}_0(r\cdot t,\neg\neg\phi)=1$.
\end{itemize}
Thus, we have $\vert\neg\neg t:\phi\rightarrow[r\cdot t]:\neg\neg\phi\vert_{\mathfrak{M}_x}=1$.

This is of course not so surprising as crisp G\"odel-Mkrtychev models correspond to classical Mkrtychev models in the respective class, and in classical modal logic we have
\[
\neg\neg\Box\theta\rightarrow\Box\neg\neg\theta\equiv\Box\theta\rightarrow\Box\theta
\]
which is of course classically realizable and this realization is thus valid in all crisp G\"odel-Mkrtychev models.

The condition $\not\vdash_{\mathcal{GJ}45_{TCS}}\neg\neg\phi$ is necessary, at least for axiomatically appropriate constant specifications as if $\vdash_{\mathcal{GJ}45_{TCS}}\neg\neg\phi$, then by internalization(Corr. \ref{cor:internalization}, as $TCS$ is axiomatically appropriate), we have $\vdash_{\mathcal{GJ}45_{TCS}}s:\neg\neg\phi$ for some $s\in Jt$ and then by propositional reasoning in $\mathcal{GJ}45_{TCS}$:
\[
\vdash_{\mathcal{GJ}45_{TCS}}\neg\neg t:\phi\rightarrow s:\neg\neg\phi.
\]
\section{Realization fails with factivity}
We can also show non-realizability for $\mathcal{GT}_\Box$ and $\mathcal{GJT}_{CS}$ as well as $\mathcal{GS}4_\Box$ and $\mathcal{GLP}_{CS}$ using the same model construction, however we need another completeness theorem for this.

This is because the factivity condition $\mathcal{E}(t,\phi)\leq\vert\phi\vert_\mathfrak{M}$ for a G\"odel-Mkrtychev model $\mathfrak{M}=\langle\mathcal{E},e\rangle$ fails for $\mathfrak{M}_x$: Per definition $\mathcal{E}_x(t,\phi)>0$ for any $t\in Jt$ and any $\phi\in\mathcal{L}_J$, thus also $\mathcal{E}_x(t,\bot)>0=\vert\bot\vert_{\mathfrak{M}_x}$.

We thus resort to an alternative definition of semantical consequence in G\"odel-Mkrtychev models. Mkrtychev, in his paper \cite{Mkr1997}, called the corresponding classical concept pre-models and our situation is quite similar to the ones in Kuznets' works \cite{Kuz2000,Kuz2008} where he resorts to pre-models as well to provide a countermodel construction in investigations into computational complexity.
\subsection{An alternative completeness theorem}
For a G\"odel-Mkrtychev model $\mathfrak{M}=\langle\mathcal{E},e\rangle$, we define at first the alternative evaluation function $\vert\cdot\vert^*_\mathfrak{M}$ as follows:
\begin{itemize}
\item $\vert\bot\vert^*_\mathfrak{M}=0$,
\item $\vert p\vert^*_\mathfrak{M}=e(p)$ for $p\in Var$,
\item $\vert\phi\rightarrow\psi\vert^*_\mathfrak{M}=\vert\phi\vert^*_\mathfrak{M}\Rightarrow\vert\psi\vert^*_\mathfrak{M}$,
\item $\vert\phi\land\psi\vert^*_\mathfrak{M}=\vert\phi\vert^*_\mathfrak{M}\odot\vert\psi\vert^*_\mathfrak{M}$,
\item $\vert t:\phi\vert^*_\mathfrak{M}=\mathcal{E}(t,\phi)\odot\vert\phi\vert^*_\mathfrak{M}$.
\end{itemize}
We write again $\mathfrak{M}\models^*\phi$ if $\vert\phi\vert^*_\mathfrak{M}=1$ and similarly for sets $\Gamma$. The corresponding definition of semantical entailment then follows naturally.
\begin{definition}
Let $\mathsf{C}$ be a class of $\mathsf{GM}$-models and $\Gamma\cup\{\phi\}\subseteq\mathcal{L}_J$. We write $\Gamma\models^*_\mathsf{C}\phi$, if for any $\mathfrak{M}\in\mathsf{C}$: $\mathfrak{M}\models^*\Gamma$ implies $\mathfrak{M}\models^*\phi$.
\end{definition}
We get the following two lemmas regarding the equivalence of the two semantics for $\mathcal{GJT}_{CS}$ and $\mathcal{GLP}_{CS}$. The lemmas and proof are fuzzy replicas of the classical cases found in \cite{Kuz2000,Mkr1997}.
\begin{lemma}\label{lem:normaltopre}
For every $\mathfrak{M}\in\mathsf{GMT}$(or $\mathsf{GMLP}$), there is a $\mathfrak{N}\in\mathsf{GM}$(or $\mathsf{GM4}$) such that $\vert\phi\vert_\mathfrak{M}=\vert\phi\vert^*_\mathfrak{N}$ for every $\phi\in\mathcal{L}_J$.
\end{lemma}
\begin{proof}
Let $\mathfrak{M}=\langle\mathcal{E},e\rangle\in\mathsf{GMT}$(or $\mathsf{GMLP}$) and set $\mathfrak{N}=\mathfrak{M}\in\mathsf{GM}$(or $\mathsf{GM4}$). We show the claim by induction on $\mathcal{L}_J$. The propositional cases are clear, so let $\phi\in\mathcal{L}_J$ such that $\vert\phi\vert_\mathfrak{M}=\vert\phi\vert^*_\mathfrak{N}$ and let $t\in Jt$. We have
\begin{align*}
\vert t:\phi\vert^*_\mathfrak{N}&=\mathcal{E}(t,\phi)\odot\vert\phi\vert^*_\mathfrak{N}\\
&=\mathcal{E}(t,\phi)\odot\vert\phi\vert_\mathfrak{M}\\
&=\mathcal{E}(t,\phi)\\
&=\vert t:\phi\vert_\mathfrak{M}
\end{align*}
where the third equality follows from the definition of $\mathsf{GMT}$(or $\mathsf{GMLP}$) which require $\mathcal{E}(t,\phi)\leq\vert\phi\vert_\mathfrak{M}$.
\end{proof}
\begin{lemma}\label{lem:pretonormal}
For every $\mathfrak{N}\in\mathsf{GM}$(or $\mathsf{GM4}$), there is a $\mathfrak{M}\in\mathsf{GMT}$(or $\mathsf{GMLP}$) such that $\vert\phi\vert^*_\mathfrak{N}=\vert\phi\vert_\mathfrak{M}$ for every $\phi\in\mathcal{L}_J$.
\end{lemma}
\begin{proof}
Let $\mathfrak{N}=\langle\mathcal{E},e\rangle\in\mathsf{GM}$(or $\mathsf{GM4}$) and define $\mathfrak{M}=\langle\mathcal{E}',e\rangle$ by $\mathcal{E}'(t,\phi)=\mathcal{E}(t,\phi)\odot\vert\phi\vert^*_\mathfrak{N}$. Then, we first show $\vert\phi\vert^*_\mathfrak{N}=\vert\phi\vert_\mathfrak{M}$ for every $\phi\in\mathcal{L}_J$ by induction on $\mathcal{L}_J$. Again, the propositional cases are clear. So let $\phi$ satisfy the claim and $t\in Jt$. Then by definition
\[
\vert t:\phi\vert_\mathfrak{M}=\mathcal{E}'(t,\phi)=\mathcal{E}(t,\phi)\odot\vert\phi\vert^*_\mathfrak{N}=\vert t:\phi\vert^*_\mathfrak{N}
\]
It remains to show that $\mathfrak{M}\in\mathsf{GMT}$. For this, we first have
\begin{align*}
\mathcal{E}'(t,\phi\rightarrow\psi)\odot\mathcal{E}'(s.\phi)&=(\mathcal{E}(t,\phi\rightarrow\psi)\odot\vert\phi\rightarrow\psi\vert^*_\mathfrak{N})\odot (\mathcal{E}(s,\phi)\odot\vert\phi\vert^*_\mathfrak{N})\\
&=(\mathcal{E}(t,\phi\rightarrow\psi)\odot\mathcal{E}(s,\phi))\odot (\vert\phi\rightarrow\psi\vert^*_\mathfrak{N}\odot\vert\phi\vert^*_\mathfrak{N})\\
&\leq\mathcal{E}(t\cdot s,\psi)\odot\vert\psi\vert^*_\mathfrak{N}\\
&=\mathcal{E}'(t\cdot s,\psi)
\end{align*}
as well as
\begin{align*}
\mathcal{E}'(t,\phi)\oplus\mathcal{E}'(s,\phi)&=(\mathcal{E}(t,\phi)\odot\vert\phi\vert^*_\mathfrak{N})\odot(\mathcal{E}(s,\phi)\odot\vert\phi\vert^*_\mathfrak{N})\\
&=(\mathcal{E}(t,\phi)\odot\mathcal{E}(s,\phi))\odot\vert\phi\vert^*_\mathfrak{N}\\
&\leq\mathcal{E}(t+s,\phi)\odot\vert\phi\vert^*_\mathfrak{N}\\
&=\mathcal{E}'(t+s,\phi).
\end{align*}
For the factivity condition, we naturally have
\[
\mathcal{E}'(t,\phi)=\mathcal{E}(t,\phi)\odot\vert\phi\vert^*_\mathfrak{N}\leq\vert\phi\vert^*_\mathfrak{N}=\vert\phi\vert_\mathfrak{M}.
\]
were the last equality follows from the before proved adequacy of $\mathfrak{M}$ for $\mathfrak{N}$. If $\mathfrak{N}$ is a $\mathsf{GM4}$-model, then also $\mathcal{E}(t,\phi)\leq\mathcal{E}(!t,t:\phi)$ and therefore
\begin{align*}
\mathcal{E}'(t,\phi)&=\mathcal{E}(t,\phi)\odot\vert\phi\vert^*_\mathfrak{N}\\
&\leq\mathcal{E}(!t,t:\phi)\odot\vert t:\phi\vert^*_\mathfrak{N}\\
&=\mathcal{E}'(!t,t:\phi)
\end{align*}
where the inequality follows from the fact that $\mathcal{E}(t,\phi)\odot\vert\phi\vert^*_\mathfrak{N}\leq\mathcal{E}(!t,t:\phi)$ as well as $\mathcal{E}(t,\phi)\odot\vert\phi\vert^*_\mathfrak{N}=\vert t:\phi\vert^*_\mathfrak{N}$.
\end{proof}
Naturally, in the two lemmas above, if one model respects a constant specification, the constructed equivalent respects it as well.

In the following, let $\mathcal{GJL}_0=\mathcal{GJT}_0,\mathcal{GLP}_0$ and $\mathsf{GMJL}=\mathsf{GM},\mathsf{GM4}$ as well as $\mathsf{GMJLT}=\mathsf{GMT},\mathsf{GMLP}$ respectively. Also, let $CS$ be a constant specification for $\mathcal{GJL}_0$. 
\begin{theorem}\label{thm:gmjltaltcomp}
For any $\Gamma\cup\{\phi\}\subseteq\mathcal{L}_J$: $\Gamma\models^*_\mathsf{GMJL_{CS}}\phi$ iff $\Gamma\vdash_{\mathcal{GJL}_{CS}}\phi$.
\end{theorem}
\begin{proof}
By the standard completeness theorem, Thm. \ref{thm:gmjlcscomp}, we show the equivalence of $\models_\mathsf{GMJLT_{CS}}$ and $\models^*_\mathsf{GMJL_{CS}}$.\\

Suppose $\Gamma\models^*_\mathsf{GMJL_{CS}}\phi$, i.e. for every $\mathfrak{M}\in\mathsf{GMJL_{CS}}$, if $\mathfrak{M}\models^*\Gamma$, then $\mathfrak{M}\models^*\phi$. By Lem. \ref{lem:normaltopre}, for every $\mathfrak{N}\in\mathsf{GMJLT_{CS}}$, if we have $\mathfrak{N}\models\Gamma$, then $\mathfrak{N}\models\phi$. Thus, we have $\Gamma\models_\mathsf{GMJLT_{CS}}\phi$.

For the reverse, suppose $\Gamma\models_\mathsf{GMJLT_{CS}}\phi$, i.e. for every $\mathfrak{M}\in\mathsf{GMJLT_{CS}}$, if $\mathfrak{M}\models\Gamma$, then $\mathfrak{M}\models\phi$. Again, now by Lem. \ref{lem:pretonormal}, for every $\mathfrak{N}\in\mathsf{GMJL_{CS}}$, if we have $\mathfrak{N}\models^*\Gamma$, then $\mathfrak{N}\models^*\phi$. Thus, we have $\Gamma\models^*_\mathsf{GMJL_{CS}}\phi$.
\end{proof}
\subsection{$\mathcal{GJT}_{CS}$ and $\mathcal{GLP}_{CS}$ do not realize $\mathcal{GT}_\Box$ and $\mathcal{GS}4_\Box$}
Here, let $TCS$ be the total constant specification for $\mathcal{GLP}_0$. Again, with $x\in (0,1)$, we define another \emph{$x$-rooted} provability model $\mathfrak{M}'_x=\langle\mathcal{E}'_x,e_x\rangle$ with $e_x$ as before and
\[
\mathcal{E}'_x(t,\phi)=\begin{cases}1 &\text{if }\vdash_{\mathcal{GLP}_{TCS}}\phi\text{ and }\vdash_{\mathcal{GLP}_{TCS}}t:\phi\\x &\text{else}\end{cases}
\]
for any $t\in Jt,\phi\in\mathcal{L}_J$. As before, we get the following lemma, however now for $TCS$ being the total constant specification for $\mathcal{GLP}_0$.
\begin{lemma}\label{lem:M'xWellDef}
For any $x\in (0,1)$, $\mathfrak{M}'_x$ is a $\mathsf{GM4_{TCS}}$-model.
\end{lemma}
As before, but with a slightly changed proof altered for the alternative semantics, we obtain the following lemma. Here, we have to restrict ourselves to propositional variables in $\mathcal{L}_J$, as in the new semantics it is relatively hard to control the truth value of compound statements containing justifications.
\begin{lemma}\label{lem:GJTRealCounterModel}
For any $p\in Var$ and any $t,s\in Jt$:
\[
\not\vdash_{\mathcal{GLP}_{TCS}}\neg\neg t: p\rightarrow s:\neg\neg p.
\]
\end{lemma}
\begin{proof}
Let $p\in Var$ and $t,s\in Jt$ as well as $x\in (0,1)$. Then naturally $\not\vdash_{\mathcal{GLP}_{TCS}}p$ and $\not\vdash_{\mathcal{GLP}_{TCS}}\neg\neg p$. Thus, $\vert t: p\vert^*_{\mathfrak{M}'_x}=\mathcal{E}'_x(t, p)\odot\vert p\vert^*_{\mathfrak{M}'_x}=\mathcal{E}'_x(t, p)\odot e_x(p)=x\odot x=x\in (0,1)$.

As $x>0$, we have $\vert\neg\neg t: p\vert^*_{\mathfrak{M}'_x}=1$ as before. However, we have 
\[
\vert s:\neg\neg p\vert^*_{\mathfrak{M}'_x}=\mathcal{E}'_x(s,\neg\neg p)\odot\vert\neg\neg p\vert^*_{\mathfrak{M}'_x}=\mathcal{E}'_x(s,\neg\neg p)\odot\sim^2e_x(p)=x\odot 1=x<1
\]
as $\not\vdash_{\mathcal{GJ}_{CS}}\neg\neg p$ and $e_x(p)=x>0$, i.e. $\sim^2e_x(p)=1$. Thus, we have
\[
\vert\neg\neg t: p\rightarrow s:\neg\neg p\vert^*_{\mathfrak{M}'_x}=x<1
\]
and as by Lem. \ref{lem:M'xWellDef}, $\mathfrak{M}'_x$ is a $\mathsf{GM4_{TCS}}$-model. Per definition, we have
\[
\not\models^*_{\mathsf{GM4_{TCS}}}\neg\neg t: p\rightarrow s:\neg\neg p,
\]
that is by Thm. \ref{thm:gmjltaltcomp}:
\[
\not\vdash_{\mathcal{GLP}_{TCS}}\neg\neg t: p\rightarrow s:\neg\neg p.
\]
\end{proof}
As before, we obtain the following two theorems.
\begin{theorem}
For any constant specification $CS$ for $\mathcal{GJT}_0$: $(Th_{\mathcal{GJT}_{CS}})^\circ\subsetneq Th_{\mathcal{GT}_\Box}$.
\end{theorem}
\begin{theorem}
For any constant specification $CS$ for $\mathcal{GLP}_0$: $(Th_{\mathcal{GLP}_{CS}})^\circ\subsetneq Th_{\mathcal{GS}4_\Box}$.
\end{theorem}
As before with $\mathcal{GJ}4_0$, also $\mathcal{GJT}_{CS}$ and $\mathcal{GLP}_{CS}$ do not even realize $\mathcal{GK}_\Box$. But in this case, the forgetful projection of the factivity axiom scheme $t:\phi\rightarrow\phi$ is not a theorem of $\mathcal{GK}_\Box$, i.e. again $(Th_{\mathcal{GJT}_{CS}})^\circ,(Th_{\mathcal{GLP}_{CS}})^\circ\not\subseteq Th_{\mathcal{GK}_\Box}$.
\section{Conclusion}
We have shown that the four G\"odel justification logics $\mathcal{GJ}_{CS}$, $\mathcal{GJT}_{CS}$, $\mathcal{GJ}4_{CS}$, $\mathcal{GLP}_{CS}$ from \cite{Pis2018} do not realize the standard G\"odel modal logics $\mathcal{GK}_\Box$, $\mathcal{GT}_\Box$, $\mathcal{GK}4_\Box$ and $\mathcal{GS}4_\Box$ from \cite{CR2009,CR2010} and by this answered one of the open problems in \cite{Gha2014,Pis2018} negatively. The G\"odel justifications logics arise as natural generalizations of the classical cases, both in model theoretic and proof theoretic terms. Also, they are compliant with the standard G\"odel modal logics via the forgetful projection. We thus advocate for the conclusion that they are not ``the wrong`` G\"odel justification logics but that there is an effective gap between G\"odel (fuzzy) justification logics and G\"odel modal logics, inherent to the many-valuedness of the base logic, which is in strong contrast to the classical counterparts where the realization theorems form one of the core natural components in their relationship, even being one of the key factors of their origination.\\

We didn't consider the G\"odel justification logics $\mathcal{GJ}45_{CS}$ and $\mathcal{GJT}45_{CS}$ explicitly which contain the negative introspection axiom $\neg t:\phi\rightarrow ?t:\neg t:\phi$. This has multiple reasons. For one, the alternative G\"odel-Mkrtychev semantics does not extend to the case of $\mathcal{GJT}45_{CS}$ which also has the factivity axiom $t:\phi\rightarrow\phi$. This is in strong similarity to classical justification logic, see. e.g. \cite{Stu2013}. For another, the development regarding G\"odel modal logics with negative introspection mostly relies on adding the possibility-modality $\Diamond$ (see e.g. \cite{CR2015}) which is also why there is no explicit statement regarding non-realization with $\mathcal{GJ}45_{CS}$ in section 2. If one would axiomatically define $\mathcal{GK}45_\Box$ (which by now seems to have been unmentioned in the literature) as the extension of $\mathcal{GK}4_\Box$ by the ($\Box$-only) negative introspection axiom scheme $\neg\Box\phi\rightarrow\Box\neg\Box\phi$, we would obtain a similar forgetful projection result as in Thm. \ref{thm:forgetfullprojection}, and thus would get the theorem
\[
(Th_{\mathcal{GJ}45_{CS}})^\circ\subsetneq Th_{\mathcal{GK}45_\Box}.
\]
directly through Lem. \ref{lem:RealCounterModel}. It shall be interesting to advance the study of G\"odel modal logics and their realizations for these cases in the future.\\

Another interesting direction is to classify the fragments of the standard G\"odel modal logics which \emph{are} realized by the basic G\"odel justification logics, if possible. Removing the axiom ($Z$) should at least provide a lower bound. Also, if possible, it would be interesting to see if there is a natural semantical characterization of this fragment which suitably generalizes G\"odel-Kripke models.
\bibliographystyle{plain}
\bibliography{ref}

\begin{thebibliography}{10}

\bibitem{Art1995}
S.~Artemov.
\newblock {Operational Modal Logic}.
\newblock Technical Report MSI 95-29, Cornell University, 1995.
\newblock Ithaca, NY.

\bibitem{Art2001}
S.~Artemov.
\newblock {Explicit Provability and Constructive Semantics}.
\newblock {\em The Bulleting of Symbolic Logic}, 7(1):1--36, 2001.

\bibitem{CR2009}
X.~Caicedo and R.~Rodriguez.
\newblock {A Godel Modal Logic}.
\newblock {\em ArXiv e-prints}, 2009.
\newblock arXiv, math.LO, 0903.2767.

\bibitem{CR2010}
X.~Caicedo and R.~Rodriguez.
\newblock {Standard G\"odel Modal Logics}.
\newblock {\em Studia Logica}, 94(2):189--214, 2010.

\bibitem{CR2015}
X.~Caicedo and R.~Rodriguez.
\newblock {Bi-modal G\"odel logic over $[0,1]$-valued Kripke frames}.
\newblock {\em Journal of Logic and Computation}, 25(1):37--55, 2015.

\bibitem{Dum1959}
M.~Dummett.
\newblock {A propositional calculus with denumerable matrix}.
\newblock {\em Journal of Symbolic Logic}, 24(2):97--106, 1959.

\bibitem{Fit2005}
M.~Fitting.
\newblock {The logic of proofs, semantically}.
\newblock {\em Annals of Pure and Applied Logic}, 132(1):1--25, 2005.

\bibitem{Gha2014}
M.~Ghari.
\newblock {Justification Logics in a Fuzzy Setting}.
\newblock {\em ArXiv e-prints}, 2014.
\newblock arXiv, math.LO, 1407.4647.

\bibitem{Gha2016}
M.~Ghari.
\newblock {Pavelka-style fuzzy justification logics}.
\newblock {\em Logic Journal of the IGPL}, 24(5):743--773, 2016.

\bibitem{Goe1932}
K.~G\"odel.
\newblock {Zum intuitionistischen Aussagenkalk\"ul}.
\newblock {\em Anzeiger der Akademie der Wissenschaften in Wien}, 69:65--66,
  1932.

\bibitem{Haj1998}
P.~H\'ajek.
\newblock {\em {Metamathematics of Fuzzy Logic}}, volume~4 of {\em Trends in
  Logic}.
\newblock Kluwer, Dordrecht, 1998.

\bibitem{Hor1969}
A.~Horn.
\newblock {Logic with Truth Values in a Linearly Ordered Heyting Algebra}.
\newblock {\em Journal of Symbolic Logic}, 34(3):395--408, 1969.

\bibitem{Kuz2000}
R.~Kuznets.
\newblock {On the complexity of explicit modal logics}.
\newblock In {\em International Workshop on Computer Science Logic,
  Proceedings}, pages 371--383. Springer Berlin Heidelberg, 2000.

\bibitem{Kuz2008}
R.~Kuznets.
\newblock {\em {Complexity Issues in Justification Logic}}.
\newblock PhD thesis, City University of New York Graduate Center, 2008.

\bibitem{Mkr1997}
A.~Mkrtychev.
\newblock {Models for the logic of proofs}.
\newblock In {\em Proceedings of Logical Foundations of Computer Science
  LFCS'97}, volume 1234 of {\em Lecture Notes in Computer Science}, pages
  266--275. Springer, 1997.

\bibitem{Pis2018}
N.~Pischke.
\newblock {A note on strong axiomatization of G\"odel Justification Logic}.
\newblock {\em ArXiv e-prints}, 2018.
\newblock arXiv, math.LO, 1809.09608.

\bibitem{Stu2013}
T.~Studer.
\newblock {Decidability for some justification logics with negative
  introspection}.
\newblock {\em Journal of Symbolic Logic}, 78(2):388--402, 2013.

\end{thebibliography}
\end{document}